\documentclass[12pt,a4paper,leqno]{amsart}

\usepackage{anysize}
\usepackage[utf8]{inputenc}
\usepackage{amsmath}
\usepackage{amssymb}
\usepackage{array}

\marginsize{2.5cm}{2.5cm}{3cm}{3cm}

\newtheorem{theorem}{Theorem}

\newtheorem{example}{Example}

\newtheorem{remark}{Remark}

\begin{document}

\title{Bilateral sums related to Ramanujan-like series}
\author{Jesús Guillera} 
\address{Univ. de Zaragoza (Spain)}
\email{jguillera@gmail.com}

\maketitle

\begin{abstract}
We define bilateral series related to Ramanujan-like series for $1/\pi^2$. Then, we conjecture a property of them and give some applications.
\end{abstract}

\section{Introduction}
In \cite{GuiRo} we constructed bilateral summations related to Ramanujan-type series for $1/\pi$ and applied them to prove a new kind of identities, which we called ``the upside-down" counterpart. In this paper we consider bilateral sums related to Ramanujan-like series for $1/\pi^2$, conjecture a property of them and give some applications. We will need the following known theorems:

\begin{theorem}\label{per} \rm
Suppose that $f(x)$ is the sum over all integers $n$ of $g(n+x)$. Then, clearly $f(x)$ is a periodic function of period $x=1$.
\end{theorem}

\begin{theorem}\label{fou} \rm
Let $f: \mathbb{C} \longrightarrow \mathbb{C}$ be a holomorphic function such that $f(x)=\mathcal{O}(e^{c \pi |{\rm Im}(x)|})$, where  $c \geq 0$ is a constant, and suppose that $f(x)$ admits a Fourier series:
\[ 
f(x)=\sum_{n=-\infty}^{+\infty} a_n e^{2 \pi i n x}. 
\]
Then $|2n| > c \, \Rightarrow a_n=0$.
\end{theorem}
\begin{proof}
It is known that the coefficients of a Fourier series are given by
\[ 
a_n=\int_{a+it}^{a+1+it} f(x) e^{-2 \pi i n x} dx. 
\]
Then for $n<0$ let $t \to+\infty$ and for $n \geq 0$ let $t \to -\infty$.
\end{proof}

\section{Ramanujan-like series for $1/\pi^2$}
Let $s_0=1/2 \,$, $s_3=1-s_1 \,$, $s_4=1-s_2 \,$. We recall that a Ramanujan-like series for $1/\pi^2$ is a series of the form
\begin{equation}\label{gui-general}
\sum_{n=0}^{\infty} \left[\prod_{i=0}^{4} \frac{(s_i)_n}{(1)_n}\right] (a+bn+cn^2) z^n =\frac{1}{\pi^2},
\end{equation}
where $z$, $a$, $b$ and $c$ are algebraic numbers and the possible couples $(s_1, s_2)$ are $(1/2,1/2)$, $(1/2,1/3)$, $(1/2,1/4)$, $(1/2,1/6)$, $(1/3,1/3)$, $(1/3,1/4)$, $(1/3,1/6)$, $(1/4,1/4)$, $(1/4,1/6)$, $(1/6,1/6)$, $(1/5,2/5)$, $(1/8,3/8)$, $(1/10,3/10)$, $(1/12,5/12)$. 
Up to date $11$ convergent and $6$ ``divergent" formulas are known in this new family (see all in the Appendix). The value
\begin{equation}\label{tau}
\tau=\frac{c}{\sqrt{1-z}},
\end{equation}
plays an important roll in the theory \cite[eq. 3.47]{Gui-matrix}.

\section{Bilateral series related to Ramanujan-like series for $1/\pi^2$}
We introduce the  function
\begin{equation}\label{rama-like-gui}
f(x)=\prod_{i=0}^{4} \frac{\cos \pi x - \cos \pi s_i}{1-\cos \pi s_i} \sum_{n \in \mathbb{Z}} (-1)^n \left[\prod_{i=0}^{4} \frac{(s_i)_{n+x}}{(1)_{n+x}}\right] [a+b(n+x)+c(n+x)^2] (-z)^{n+x}.
\end{equation}
Applying Theorem \ref{per}, we see that it is a periodic function of period $x=1$ because $\cos \pi x - \cos \pi s_0 =\cos \pi x$ and $(-1)^n \cos \pi x = \cos \pi (n+x)$. Besides it is analytic due to the factor with the cosenus, which cancels the poles at $x=-s_i$ and therefore all the other poles due to the periodicity of $f(x)$. On the other hand we conjecture that $f(x)=\mathcal{O}(e^{5 \pi |{\rm Im}(x)|})$. Hence by Theorem \ref{fou} we deduce that 
\begin{multline}\label{bilateral-general}
\prod_{i=0}^{4} \frac{\cos \pi x - \cos \pi s_i}{1-\cos \pi s_i} \sum_{n \in \mathbb{Z}} (-1)^n \left[\prod_{i=0}^{4} \frac{(s_i)_{n+x}}{(1)_{n+x}}\right] [a+b(n+x)+c(n+x)^2] (-z)^{n+x} \\
= \frac{1}{\pi^2} \left( \frac{}{}  u_1 \cos 2\pi x + u_2 \cos 4\pi x + (1-u_1-u_2) + v_1 \sin 2 \pi x + v_2 \sin 4 \pi x \right),
\end{multline}
where $u_1$, $u_2$, $v_1$, $v_2$ do not depend on $x$, and we can determine their values by giving values to $x$. We conjecture that for the bilateral series corresponding to a Ramanujan-like series for $1/\pi^2$ (that is, when $z$, $a$, $b$, $c$ are algebraic numbers), the values of $u_1$, $u_2$, $v_1$ and $v_2$ are rational numbers. In addition, for the case of alternating series we conjecture that $v_1=v_2=0$. These bilateral series are ``divergent", because the sum to the left or the sum to the right ``diverges". However we turn them convergent by analytic continuation. We can write (\ref{bilateral-general}) in the following way:
\begin{multline}\label{bilateral-general-split}
\sum_{n=0}^{\infty} (-1)^n \left[\prod_{i=0}^{4} \frac{(s_i+x)_{n}}{(1+x)_{n}} \right] [a+b(n+x)+c(n+x)^2] (-z)^{n+x} \\ + x^5 \sum_{n=1}^{\infty} (-1)^n \left[\prod_{i=0}^{4} \frac{(1-x)_{n}}{(s_i-x)_{n}} \right] \frac{a+b(-n+x)+c(-n+x)^2}{(-n+x)^5} (-z)^{-n+x} = \\
\frac{1}{\pi^2} \left[\prod_{i=0}^{4} \frac{(1)_x}{(s_i)_{x}} \right] \frac{u_1 \cos 2\pi x + u_2 \cos 4\pi x + (1-u_1-u_2) + v_1 \sin 2 \pi x + v_2 \sin 4 \pi x}{\csc^2 \pi s_1 \csc^2 \pi s_2 \cos \pi x (\cos^2 \pi x - \cos^2 \pi s_1)(\cos^2 \pi x - \cos^2 \pi s_2)},
\end{multline}
where we have splitted the summation in two sums and used the properties
\[
\frac{1}{(1)_{-n+x}}=(-1)^{n+1} \frac{(1-x)_n}{(1)_x} \frac{x}{n-x}, \quad 
(s_i)_{-n+x}=(-1)^n \frac{(s_i)_x}{(1-s_i-x)_n}.
\]
As the second sum is $\mathcal{O}(x^5)$, formula (\ref{bilateral-general-split}) implies an expansion of the form
\[
\sum_{n=0}^{\infty} (-1)^n \left[\prod_{i=0}^{4} \frac{(s_i+x)_{n}}{(1+x)_{n}} \right] [a+b(n+x)+c(n+x)^2] (-z)^{n+x} = \frac{1}{\pi^2} - k \frac{x^2}{2!} + j \pi^2 \frac{x^4}{4!}+ \mathcal{O}(x^5),
\]
if $z<0$ (alternating series), and an expansion of the form
\[
\sum_{n=0}^{\infty} \left[\prod_{i=0}^{4} \frac{(s_i+x)_{n}}{(1+x)_{n}} \right] [a+b(n+x)+c(n+x)^2] z^{n+x} = \frac{1}{\pi^2} - k \frac{x^2}{2!} + j \pi^2 \frac{x^4}{4!}+ \mathcal{O}(x^5),
\]
if $0<z<1$ (convergent series of positive terms). 
\par Hence, the conjecture that the coefficients of the Fourier expansion are rational is equivalent to the conjecture that $k$ and $j$ are rational stated in \cite{Gui-matrix,AlGu}. The relation of $\tau$ with $k$ and $j$ is \cite[eq. 3.48]{Gui-matrix}:
\[
\tau^2=\frac{j}{12}+\frac{k^2}{4}+\frac{5k}{3}+1+(\cot^2 \pi s_1)(\cot^2 \pi s_2)+(1+k)(\cot^2 \pi s_1+\cot^2 \pi s_2).
\]
From (\ref{bilateral-general-split}) we deduce that
\begin{equation}
\sum_{n \in \mathbb{Z}} \left[\prod_{i=0}^{4} \frac{(s_i)_n}{(1)_n}\right] (a+b n+c n^2) 
z^n=
\sum_{n=0}^{\infty} \left[\prod_{i=0}^{4} \frac{(s_i)_n}{(1)_n}\right] (a+b n+c n^2) z^n,
\end{equation}
and as the function (\ref{rama-like-gui}) has period $x=1$, we see that for all integer $x$, we have
\begin{multline}\label{x-integer}
\sum_{n \in \mathbb{Z}} (-1)^n \left[\prod_{i=0}^{4} \frac{(s_i)_{n+x}}{(1)_{n+x}} \right] [a+b(n+x)+c(n+x)^2] (-z)^{n+x} \\ = (-1)^x \sum_{n=0}^{\infty} \left[\prod_{i=0}^{4} \frac{(s_i)_n}{(1)_n}\right] (a+b n+c n^2) z^n.
\end{multline}
If $z>0$, we can write
\begin{multline}\label{z-mayor0}
\sum_{n \in \mathbb{Z}} (-1)^n \left[\prod_{i=0}^{4} \frac{(s_i)_{n+x}}{(1)_{n+x}} \right] [a+b(n+x)+c(n+x)^2] (-z)^{n+x} \\ =e^{-i \pi x} \sum_{n \in \mathbb{Z}} \left[\prod_{i=0}^{4} \frac{(s_i)_{n+x}}{(1)_{n+x}} \right] [a+b(n+x)+c(n+x)^2] z^{n+x}.
\end{multline}
Taking into account (\ref{x-integer}) and (\ref{z-mayor0}), we have written the following Maple procedure: 
\begin{verbatim}
bilater:=proc(s1,s2,z0,x,j)
local p,s0,s3,s4,z1,h,hh,r,u,v;
p:=(a,b)->pochhammer(a,b): s0:=1/2; s3:=1-s1; s4:=1-s2:
h:=n->p(s0,n)*p(s1,n)*p(s2,n)*p(s3,n)*p(s4,n)/p(1,n)^5:
if type(x,integer) then; hh:=n->h(n)*n^j:
r:=evalf((-1)^x*subs(z=z0,sum(hh(n)*z^n,n=0..infinity))):
return r; else;
if z0<0 then; z1:=z0; hh:=(j,n)->(-1)^n*h(n+x)*(n+x)^j; 
else z1:=-z0; hh:=(j,n)->exp(-I*Pi*x)*h(n+x)*(n+x)^j; fi;
u:=j->subs(z=z1,sum(hh(j,n)*(-z)^(n+x),n=0..infinity)):
v:=j->subs(z=z1,sum(hh(j,-n)*(-z)^(-n+x),n=1..infinity)):
r:=evalf(u(j)+v(j)): return r: fi;
end:
\end{verbatim}
This procedure calculates the bilateral sum
\[
\sum_{n \in \mathbb{Z}} (-1)^n \left[\prod_{i=0}^{4} \frac{(s_i)_{n+x}}{(1)_{n+x}} \right] (-z)^{n+x} \, (n+x)^j,
\] 
and helps the reader to check the examples.

\section{Examples of bilateral series}
We give some examples of bilateral series related to Ramanujan-like series for $1/\pi^2$.

\begin{example} \rm
Formula (\ref{f3}) in the Appendix is
\begin{equation}\label{gui-1}
\frac18 \sum_{n=0}^{\infty} (-1)^n \frac{(\frac12)_n^5}{(1)_n^5}\frac{20n^2+8n+1}{4^n} = \frac{1}{\pi^2}. 
\end{equation}
Taking $v_1=v_2=0$ and giving two values to $x$ we get numerical approximations of $u_1$, $u_2$. We observe that these values look rational and do not change using other values of $x$. Hence,
we conjecture that
\begin{equation}\label{Z-gui-1}
\frac18 \sum_{n \in \mathbb{Z}} (-1)^n \frac{(\frac12)_{n+x}^5}{(1)_{n+x}^5}\frac{20(n+x)^2+8(n+x)+1}{4^{n+x}}= \frac{1-\frac12 \cos 2\pi x + \frac12 \cos 4\pi x}{\pi^2 \cos^5 \pi x},
\end{equation}
after replacing $u_1$, $u_2$ with their guessed rational values. 
\end{example}

\begin{example}\label{other-example} \rm
Formula (\ref{f10}) in the Appendix is
\[
\sum_{n=0}^{\infty} (-1)^n \frac{\left(\frac12\right)_n\left(\frac13\right)_n \left(\frac23\right)_n\left(\frac16\right)_n\left(\frac56\right)_n}{(1)_n^5} \left( \frac34 \right)^{6n} (1930n^2+549n+45) = \frac{384}{\pi}.
\]
From it we conjecture the bilateral form
\begin{multline}
\frac{1}{384} \sum_{n \in \mathbb{Z}} (-1)^n \frac{\left(\frac12\right)_{n+x}\left(\frac13\right)_{n+x}\left(\frac23\right)_{n+x}\left(\frac16\right)_{n+x}\left(\frac56\right)_{n+x}}{(1)_{n+x}^5}  \left( \frac34 \right)^{6(n+x)} \\ \times \left( 1930(n+x)^2+549(n+x)+45 \right) = 
\frac{11-14 \cos 2\pi x + 6 \cos 4\pi x}{\pi^2 \cos \pi x \, (4\cos^2 \pi x-1)(4\cos^2 \pi x -3)},
\end{multline}
after identifying the coefficients.
\end{example}

\begin{example} \rm
From the following formula:
\begin{equation}\label{gui-16}
\sum_{n=0}^{\infty} \frac{\left(\frac12\right)_n^3\left(\frac14\right)_n\left(\frac34\right)_n}{(1)_n^5}\frac{120n^2+34n+3}{16^n} = \frac{32}{\pi^2},
\end{equation}
which is (\ref{f2}) in the Appendix, we get the bilateral form
\begin{multline*}
\frac{1}{32} \sum_{n \in \mathbb{Z}} \frac{\left(\frac12\right)_{n+x}^3\left(\frac14\right)_{n+x}\left(\frac34\right)_{n+x}}{(1)_{n+x}^5}  \left( \frac{1}{16} \right)^{n+x} (120(n+x)^2+34(n+x)+3) = \\
e^{i \pi x} \, \frac{3-\frac72 \cos 2 \pi x + \frac32 \cos 4 \pi x + \left( \frac12 \sin 2\pi x - \frac12 \sin 4 \pi x \right) i}{\pi^2 \cos^3 \pi x \, (2\cos^2 \pi x -1)},
\end{multline*}
after identifying the coefficients of the Fourier expansion from their numerical approximations. Observe that $v_1$ and $v_2$ are not null because the series in (\ref{gui-16}) is of positive terms. Observe also the factor $e^{i \pi x}$, which comes from $(-1)^n / (-16)^{n+x}$
\end{example}

\begin{example} \rm
For the formula (\ref{f7}) in the Appendix we have the bilateral identity
\begin{multline*}
\frac{\sqrt 7}{392} \sum_{n \in \mathbb{Z}}  \frac{\left(\frac12 \right)_{n+x} \left( \frac18 \right)_{n+x} \left( \frac38 \right)_{n+x} \left( \frac58 \right)_{n+x} \left( \frac78 \right)_{n+x}}{(1)_{n+x}^5} \frac{1}{7^{4(n+x)}} (1920(n+x)^2+304(n+x)+15) \\ =
e^{i \pi x} \,  \frac{ 29-\frac{79}{2} \cos 2 \pi x + \frac{23}{2} \cos 4 \pi x + \left( \frac52 \sin 2\pi x - \frac32 \sin 4 \pi x \right) i }{\pi^2 \csc^2 \frac{\pi}{8} \csc^2 \frac{3\pi}{8}  \cos \pi x \left( \cos^2 \pi x - \cos^ 2 \frac{\pi}{8} \right)\left( \cos^2 \pi x - \cos^ 2 \frac{3\pi}{8} \right)},
\end{multline*}
after identifying the coefficients.
\end{example}

\begin{example} \rm
Looking at the formula (\ref{f9}) in the Appendix which has $z=(3/\phi)^3$, we had the intuition that another series with the other sign of the square root, namely $z=-(3\phi)^3$, could exist and we were right because we discovered it by using the PSLQ algorithm. Here we recover it in a different way by writing the bilateral identity
\begin{multline}\label{diverg-phi}
\sum_{n \in \mathbb{Z}} \frac{\left(\frac12\right)_{n+x}^3\left(\frac13\right)_{n+x}\left(\frac23 \right)_{n+x}}{(1)_{n+x}^5}  
(-1)^n \left(3 \phi\right)^{3(n+x)} (c(n+x)^2+b(n+x)+a) = \\
\frac{3}{4\pi^2} \frac{u_1 \cos 2 \pi x + u_2 \cos 4 \pi x + (1-u_1-u_2)}{\cos^3 \pi x \left(\cos^2 \pi x - \frac14 \right)}.
\end{multline}
Giving five values to $x$ we get numerical approximations of $a$, $b$, $c$, $u_1$ and $u_2$, which we could identify:
\begin{equation}\label{sol-diverg-phi}
u_1=\frac{17}{36}, \quad u_2=\frac{3}{16}, \quad c=2408 + \frac{216}{\phi}, \quad b=1800 + \frac{162}{\phi}, \quad a= 333 + \frac{30}{\phi}.
\end{equation}
Replacing these values and taking $x=0$, we arrive at the ``divergent" formula (\ref{d5}), in the second list of the Appendix.
\end{example}

\begin{example} \rm 
From the ``divergent" series (\ref{d3}) in the Appendix, namely 
\[
\sum_{n=0}^{\infty}\frac{\left(\frac12 \right)_n^3 \left( \frac13 \right)_n \left( \frac23 \right)_n}{(1)_n^5} (28n^2+18n+3) (-1)^n 3^{3n} = \frac{6}{\pi^2},
\]
we have the following bilateral sum evaluation:
\begin{multline}\label{Z-gui-27}
\frac16 \sum_{n \in \mathbb{Z}} \frac{(\frac12)_{n+x}^3(\frac13)_{n+x}(\frac23)_{n+x}}{(1)_{n+x}^5} \left[ 28(n+x)^2+18(n+x)+3 \right] \, (-1)^n \, 27^{n+x} \\ 
= \frac{5 + 4 \cos 2\pi x + 3 \cos 4\pi x}{4 \pi^2 \cos^3 \pi x (4\cos^2 \pi x - 1)},
\end{multline}
after identifying the coefficients of the Fourier terminating expansion from their numerical approximations obtained giving values to $x$.
\end{example}

\section{Applications of the bilateral series}
If $|z|>1$, then letting $x \to -1/2$ or $x \to -s_1$, etc in (\ref{bilateral-general-split}), we obtain the evaluation of some convergent series. In general, taking the limit of (\ref{bilateral-general-split}) as $x \to -s_j$, where $j \in \{0,1,2,3,4\}$, we deduce the following identity:
\begin{multline}\label{diverg-to converg}
\sum_{n=0}^{\infty} \left[\prod_{i=0}^{4} \frac{(s_j)_{n}}{(s_i+s_j)_{n}} \right] 
\left[ a+b(-n-s_j)+c(-n-s_j)^2 \right] z^{-n} =  \\
\frac{(-z)^{s_j}}{\pi^2}\lim_{x \to -s_j}\left[\prod_{i=0}^{4} \frac{(1)_x}{(s_i)_{x}} \right] 
\frac{u_1 \cos 2\pi x + u_2 \cos 4\pi x + (1-u_1-u_2) + v_1 \sin 2 \pi x + v_2 \sin 4 \pi x}{\csc^2 \pi s_1 \, \csc^2 \pi s_2 \, \cos \pi x (\cos^2 \pi x - \cos^2 \pi s_1)(\cos^2 \pi x - \cos^2 \pi s_2)}.
\end{multline}
Another application is explained in \cite{GuiRo}, but only proved for bilateral sums related to Ramanujan-type series for $1/\pi$.

\begin{example} \rm
From the formula (\ref{Z-gui-27}), and  taking $s_j=1/2$ in (\ref{diverg-to converg}), we get
\begin{equation}
\sum_{n=0}^{\infty} \frac{\left(\frac12\right)_n^5}{\left(1\right)_n^3 \left(\frac16\right)_n \left(\frac56 \right)_n} \frac{28n^2+10n+1}{6n+1} \left( \frac{-1}{27} \right)^n = \frac{3}{\pi},
\end{equation}
which is  a new convergent series for $1/\pi$. In the same way, from the ``divergent" series (\ref{diverg-phi}-\ref{sol-diverg-phi}), we get the convergent evaluation
\begin{equation}
\sum_{n=0}^{\infty} \frac{\left(\frac12\right)_n^5}{\left(1\right)_n^3 \left(\frac16\right)_n \left(\frac56 \right)_n}  \frac{ (2408  \! + \!  \frac{216}{\phi}) n^2 \! + \!  (608 + \frac{54}{\phi})n +  (35 + \frac{3}{\phi})}{6n+1} \left( \frac{-1}{3\phi} \right)^{3n} = \frac{3\sqrt{\phi^3}}{\pi},
\end{equation}
where $\phi$ is the fifth power of the golden ratio.
\end{example}

\begin{example} \rm
From (\ref{Z-gui-1}), and letting $s_j \to -1/2$ in (\ref{diverg-to converg}), we recover formula (\ref{d1}) in the Appendix. We observe that a formula with $z$ implies another one with $z^{-1}$ in the family $s_1=s_2=1/2$, and we will refer to this property as ``duality".
\end{example}

\section{The mirror map}
All the parameters of the bilateral sums related to Ramanujan-like series for $1/\pi^2$ are algebraic, but the value of $q$ (related to $z$ by the mirror map) is not. Now we will define a function $w(x)$ in which $q$ is easily related to the coefficients of the Fourier expansion. First, we let $u(x)$ and $v(x)$ be the analytic functions
\[
u(x) =\prod_{i=0}^{4} \frac{\cos \pi x - \cos \pi s_i}{1-\cos \pi s_i} \sum_{n \in \mathbb{Z}} (-1)^n \left[\prod_{i=0}^{4} \frac{(s_i)_{n+x}}{(1)_{n+x}}\right] (-z)^{n+x}, 
\]
and
\[
v(x) =\prod_{i=0}^{4} \frac{\cos \pi x - \cos \pi s_i}{1-\cos \pi s_i} \sum_{n \in \mathbb{Z}} (-1)^n \left[\prod_{i=0}^{4} \frac{(s_i)_{n+x}}{(1)_{n+x}}\right] (n+x) (-z)^{n+x}.
\]
Then, we define the analytic function 
\[ 
w(x) = v(0)u(x)-u(0)v(x).
\]
We have the following Fourier expansion:
\[
w(x)=\frac{}{}  a_1 \cos 2\pi x + a_2 \cos 4\pi x + (1-a_1-a_2) + b_1 \sin 2 \pi x + b_2 \sin 4 \pi x,
\]
The above identity for the function $w(x)$ implies an expansion of the form
\begin{multline}\label{q-alternating}
\sum_{n=0}^{\infty} \left[\prod_{i=0}^{4} \frac{(s_i)_{n}}{(1)_{n}} \right] n z^n \sum_{n=0}^{\infty} (-1)^n \left[ \prod_{i=0}^{4} \frac{(s_i)_{n+x}}{(1)_{n+x}}\right] (-z)^{n+x} \\
- \sum_{n=0}^{\infty} \left[\prod_{i=0}^{4} \frac{(s_i)_{n}}{(1)_{n}} \right] z^n \sum_{n=0}^{\infty} (-1)^n \left[ \prod_{i=0}^{4} \frac{(s_i)_{n+x}}{(1)_{n+x}}\right] (n+x) (-z)^{n+x}
\\ = p_1 x + p_2 x^2 + p_3 x^3 + p_4 x^4 + \mathcal{O}(x^5)
\end{multline}
if $z<0$ (alternating series), and an expansion of the form
\begin{multline}\label{q-positive}
\sum_{n=0}^{\infty} \left[\prod_{i=0}^{4} \frac{(s_i)_{n}}{(1)_{n}} \right] n z^n \sum_{n=0}^{\infty} \left[ \prod_{i=0}^{4} \frac{(s_i)_{n+x}}{(1)_{n+x}}\right] z^{n+x} \\
- \sum_{n=0}^{\infty} \left[\prod_{i=0}^{4} \frac{(s_i)_{n}}{(1)_{n}} \right] z^n \sum_{n=0}^{\infty} \left[ \prod_{i=0}^{4} \frac{(s_i)_{n+x}}{(1)_{n+x}}\right] (n+x) z^{n+x}
\\ = p_1 x + p_2 x^2 + p_3 x^3 + p_4 x^4 + \mathcal{O}(x^5)
\end{multline}
if $z>0$ and $z \neq 1$ (series of positive terms). From the differential equations for Calabi-Yau threefolds \cite{YaZu, Gui-matrix, AlGu, Zu4}, we deduce that
\[ 
-\frac{1}{\pi} \frac{p_2}{p_1}=t, \quad q=-e^{-\pi t} \quad \text{or} \quad \, q=e^{-\pi t}
\]
for (\ref{q-alternating}) and (\ref{q-positive}) respectively, where $z=z(q)$ is the mirror map, and 
\[
k = 2 \left( \frac{1}{\pi^2} \frac{p_4}{p_2} - \frac53 - \cot^2 \pi s_1 - \cot^2 \pi s_2 \right).
\]
The method used in \cite{AlGu} is good for convergent Ramanujan series, but not when the series is ``too divergent". However the method based on bilateral series permits to calculate $q$ with many digits in all the cases, when we know the value of $z$.

\begin{example} \rm  
Taking $s_1=1/2$, $s_2=1/3$ and $z=27\phi^{-3}$ (formula (\ref{f9}) in the Appendix), and using the formulas in \cite{AlGu}, and $q=e^{-\pi t}$, we get
\[
t_1=t(27 \phi^{-3})=3.619403396730928522140860042453285904901.
\]
However for $z=(-3 \phi)^3$ (formula (\ref{d5}) in the Appendix), we need to use the method based on bilateral sums. Then, we  get that the value of $t$ corresponding to (\ref{d5}), is with $40$ correct digits (we can calculate many more) equal to
\[
t_2=t(-27\phi^3)=1.233412165189594043723756118628903242841.
\]
We get the values of $\tau$ from (\ref{tau}), and they are $\tau_1=4/3 \cdot \sqrt{10}$ and $\tau_2=1/9 \cdot \sqrt{10}$. We thought that the values of $t_1$ and $t_2$ would be easily related but we have not succeeded in finding any relation.
\end{example}

\begin{example} \rm  
For $s_1=s_2=1/2$, the series with $z$ and with $z^{-1}$ satisfy duality. For the formulas (\ref{f1}) and (\ref{d2}) in the Appendix, we have
\begin{align*}
t_1=t(-2^{-10})   &=4.412809109031200738238212268698423552548,
\\ t_2=t(-2^{10}) &=1.252302434231184754606061614044396018505.
\end{align*}
We get $\tau_1=\tau(-2^{-10})=\sqrt{41}$ and $\tau_2=1/16 \cdot \sqrt{41}$ from (\ref{tau}), and we observe that
\[
t_1 t_2 = 2 \frac{\sqrt{41}+3}{\sqrt{41}-3}.
\]
Hence the values of $t_1$ and $t_2$ are nicely related.
\end{example}

\section{Conclusion}
We have conjectured that the terminating Fourier expansion of the function (\ref{rama-like-gui}) corresponding to a Ramanujan-like series for $1/\pi^2$ has rational coefficients. A challenging problem is to prove it. That is, to obtain rigourosly the bilateral form corresponding to any Ramanujan-like series for $1/\pi^2$. This would suppose to give another step towards understanding the family of formulas for $1/\pi^2$.
\par Below, we show how to solve the problem for the Ramanujan-type series for $1/\pi$. Consider, for example the Ramanujan series
\begin{equation}\label{rama-882}
\sum_{n=0}^{\infty} (-1)^n \frac{\left(\frac12\right)_n\left(\frac14\right)_n \left(\frac34\right)_n}{(1)_n^3} \frac{21460n+1123}{882^{2n}} = \frac{3528}{\pi}.
\end{equation}
To solve the problem of finding the bilateral form we need to determine the value of $u$ in
\begin{equation}\label{bilateral-rama-882}
\sum_{n \in \mathbb{Z}} (-1)^n \frac{\left(\frac12\right)_{n+x} \left(\frac14\right)_{n+x} \left(\frac34\right)_{n+x}}{(1)_{n+x}^3} \frac{21360(n+x)+1123}{3528 \cdot 882^{2(n+x)}} = \frac{1-u+u\cos 2 \pi x}{\pi \cos \pi x (2\cos^2 \pi x-1)}.
\end{equation}
Expanding the right side, comparing with \cite[Expansion 1.1]{Gui-matrix}, and using \cite[eq. 2.30-2.32]{Gui-matrix}, we rigorously obtain that $u=10$,
which is a rational number.

\section{Appendix}

In the following two lists $\phi$ means the fifth power of the golden ratio. That is
\[
\phi = \left (\frac{1+\sqrt 5}{2} \right)^5.
\]
The notation ${\overset{?}=}$ means that the formula remains unproved, and the notation $``="$ means that we get the equality by analytic continuation.

\subsection{List of convergent formulas}

\begin{equation}\label{f1}
\sum_{n=0}^{\infty}\frac{\left(\frac12 \right)_n^5}{(1)_n^5}\frac{(-1)^n}{2^{10n}} (820n^2+180n+13)=\frac{128}{\pi^2},
\end{equation}
\begin{equation}\label{f2}
\sum_{n=0}^{\infty}  \frac{\left( \frac12 \right)_n^3 \left( \frac14 \right)_n
\left( \frac34 \right)_n}{(1)_n^5} \frac{1}{2^{4n}} (120n^2+34n+3)=\frac{32}{\pi^2},
\end{equation}
\begin{equation}\label{f3}
 \sum_{n=0}^{\infty}  \frac{\left( \frac12 \right)_n^5}{(1)_n^5} \frac{(-1)^n}{2^{2n}} (20n^2+8n+1)=\frac{8}{\pi^2},
\end{equation}
\begin{equation}\label{f4}
\sum_{n=0}^{\infty}  \frac{\left( \frac12 \right)_n \left( \frac14 \right)_n \left( \frac34 \right)_n \left( \frac16 \right)_n \left( \frac56 \right)_n}{(1)_n^5} \frac{(-1)^n}{2^{10n}} (1640n^2+278n+15)\, {\overset{?}=} \, \frac{256 \sqrt 3}{\pi^2},
\end{equation}
\begin{equation}\label{f5}
\sum_{n=0}^{\infty} \frac{\left( \frac12 \right)_n \left( \frac13 \right)_n \left( \frac23 \right)_n \left( \frac14 \right)_n \left( \frac34 \right)_n}{(1)_n^5}  \frac{(-1)^n}{48^n} (252n^2+63n+5) \, {\overset{?} =} \, \frac{48}{\pi^2},
\end{equation}
\begin{equation}\label{f6}
\sum_{n=0}^{\infty}  \frac{\left( \frac12 \right)_n \left( \frac13 \right)_n \left( \frac23 \right)_n \left( \frac16 \right)_n \left(\frac56 \right)_n}{(1)_n^5} \frac{(-1)^n}{80^{3n}} (5418n^2+693n+29)\, {\overset{?} =} \, \frac{128 \sqrt 5}{\pi^2},
\end{equation}
\begin{equation}\label{f7}
\sum_{n=0}^{\infty}  \frac{\left(\frac12 \right)_n \left( \frac18 \right)_n \left( \frac38 \right)_n \left( \frac58 \right)_n \left( \frac78 \right)_n}{(1)_n^5} \frac{1}{7^{4n}} (1920n^2+304n+15) \, {\overset{?} =} \, \frac{56 \sqrt 7}{\pi^2},
\end{equation}
\begin{equation}\label{f8}
\sum_{n=0}^{\infty}  \frac{\left(\frac12 \right)_n^3 \left( \frac13 \right)_n \left( \frac23 \right)_n}{(1)_n^5} \left( \frac34 \right)^{3n} (74n^2+27n+3)\, = \, \frac{48}{\pi^2},
\end{equation}
\begin{equation}\label{f9}
\sum_{n=0}^{\infty} \frac{\left(\frac12 \right)_n^3 \! \left( \frac13 \right)_n \! \left( \frac23 \right)_n}{(1)_n^5} \! \left( \frac{3}{\phi} \right)^{3n} \! \left[ (32 \! - \! \frac{216}{\phi}) n^2 \! + \! (18 \! - \! \frac{162}{\phi})n \! + \! (3 \! -\! \frac{30}{\phi}) \right] \, {\overset{?} =} \, \frac{3}{\pi^2},
\end{equation}
\begin{equation}\label{f10}
\sum_{n=0}^{\infty}  \frac{\left(\frac12 \right)_n \left( \frac13 \right)_n \left( \frac23 \right)_n \left( \frac16 \right)_n \left( \frac56 \right)_n}{(1)_n^5} (-1)^n \left( \frac{3}{4} \right)^{6n} (1930n^2+549n+45) \, {\overset{?}=} \, \frac{384}{\pi^2},
\end{equation}
\begin{equation}\label{f11}
\sum_{n=0}^{\infty}  \frac{\left(\frac12 \right)_n \left( \frac13 \right)_n \left( \frac23 \right)_n \left( \frac16 \right)_n \left( \frac56 \right)_n}{(1)_n^5} \left( \frac{3}{5} \right)^{6n} (532n^2+126n+9) \, {\overset{?}=} \, \frac{375}{4\pi^2}.
\end{equation}

\subsection{List of ``divergent" formulas}

\begin{equation}\label{d1}
\sum_{n=0}^{\infty}\frac{\left(\frac12 \right)_n^5}{(1)_n^5} (10n^2+6n+1) (-1)^n 4^n \, \, ``\!=\!" \, \, \frac{4}{\pi^2},
\end{equation}
\begin{equation}\label{d2}
\sum_{n=0}^{\infty}\frac{\left(\frac12 \right)_n^5}{(1)_n^5} (205n^2+160n+32) (-1)^n 2^{10n} \, \, ``\!=\!" \, \, \frac{16}{\pi^2},
\end{equation}
\begin{equation}\label{d3}
\sum_{n=0}^{\infty}\frac{\left(\frac12 \right)_n^3 \left( \frac13 \right)_n \left( \frac23 \right)_n}{(1)_n^5} (28n^2+18n+3) (-1)^n 3^{3n} \, `` \, = \, " \, \frac{6}{\pi^2},
\end{equation}
\begin{equation}\label{d4}
\sum_{n=0}^{\infty}\frac{\left(\frac12 \right)_n \left( \frac13 \right)_n \left( \frac23 \right)_n\left( \frac14 \right)_n \left( \frac34 \right)_n}{(1)_n^5} (172n^2+75n+9) (-1)^n \left( \frac{27}{16} \right)^n \, \, ``=" \, \, \frac{48}{\pi^2},
\end{equation}
\begin{equation}\label{d5}
\sum_{n=0}^{\infty} \frac{\left(\frac12 \right)_n^3 \! \left( \frac13 \right)_n \! \left( \frac23 \right)_n}{(1)_n^5} \! (-3 \phi)^{3n} \! \left[ (2408 \! + \! \frac{216}{\phi}) n^2 \! + \! (1800 \! + \! \frac{162}{\phi})n \! + \! (333 \! + \! \frac{30}{\phi}) \right] \, `` \, {\overset{?}=} \, " \, \frac{36}{\pi^2},
\end{equation}
\begin{equation}\label{d6}
\sum_{n=0}^{\infty}\frac{\left(\frac12 \right)_n \left( \frac15 \right)_n \left( \frac25 \right)_n\left( \frac35 \right)_n \left( \frac45 \right)_n}{(1)_n^5} (483n^2+245n+30) (-1)^n \left( \frac{5^5}{2^8} \right)^n \, ``\, {\overset{?}=}\," \,  \frac{80}{\pi^2}.
\end{equation}
Formulas (\ref{f1}, \ref{f2}, \ref{f3}, \ref{f8}, \ref{d1}), were proved by the author using the WZ-method. For a proof of (\ref{d2}), (\ref{d3}) and (\ref{d4}) see \cite[Section 6]{gui-dougall}. All the other formulas are conjectured. The conjectured formula (\ref{f11}) is joint with Gert Almkvist. In \cite{AlGu} we recovered all the known convergent series for $1/\pi^2$ and also two ``divergent" ones: (\ref{d1}) and (\ref{d4}). I discovered formula (\ref{d5}) by replacing $\sqrt{5}$ with $-\sqrt{5}$ in the value of $z$ of (\ref{f9}) and then using the PSLQ algorithm. We can check that the mosaic supercongruences pattern \cite{Gu6} holds for the formula (\ref{d5}). Inspired by these congruences, by the conjectured formulas \cite[Conj. 1.1--1.6]{Sun} and by \cite{GuiRo}, we observe that 
\begin{align}
\nonumber \sum_{n=1}^{\infty}  \frac{(1)_n^5}{\left(\frac12 \right)_n^3 \left( \frac13 \right)_n \left( \frac23 \right)_n} \left( \frac{-1}{3\phi} \right)^{3n} 
& \frac{(2408+216 \phi^{-1})n^2-(1800+162 \phi^{-1})n+(333+30\phi^{-1})}{n^5} \\ {\overset{?} =}
& \frac{1125}{4}\sqrt{5}L_5(3)-448\zeta(3),
\end{align}
seems true. As the convergence is fast, we can use it to get many digits of $L_5(3)$. 
I discovered the formula (\ref{d6}) very recently. Related to it are the Zudilin-type supercongruences \cite{Zu0, GuiZu}:
\begin{equation}\label{sc1}
\sum_{n=0}^{p-1}\frac{\left(\frac12 \right)_n \left( \frac15 \right)_n \left( \frac25 \right)_n\left( \frac35 \right)_n \left( \frac45 \right)_n}{(1)_n^5} (483n^2+245n+30) (-1)^n \left( \frac{5^5}{2^8} \right)^n \, \, {\overset{?} \equiv} \, \, 30 p^3 \pmod{p^5},
\end{equation}
for primes $p>2$, and the ``upside-down" counterpart \cite{GuiRo}, namely 
\begin{equation}\label{ud1}
\sum_{n=1}^{\infty}\frac{(1)_n^5}{\left(\frac12 \right)_n \left( \frac15 \right)_n \left( \frac25 \right)_n\left( \frac35 \right)_n \left( \frac45 \right)_n} \frac{-483n^2+245n-30}{n^5} (-1)^n \left( \frac{2^8}{5^5} \right)^n \, \, {\overset{?}=} \, \, 896 \, \zeta(3),
\end{equation}
which is a convergent series for $\zeta(3)$.

\subsection{List of Ramanujan-like formulas of higher degree (added in June 2018)}

B. Gourevitch, inspired by the formulas in \cite{guilleraAAMwz}, searched similar formulas for $1/\pi^3$ with the help of the PSLQ algorithm, and discovered the following one \cite{guilleraEMpi2}:
\begin{equation}\label{gourevitch}
\sum_{n=0}^{\infty}\frac{\left(\frac12 \right)_n^7}{(1)_n^7}\frac{1}{2^{6n}} (168n^3 +76n^2+14n+1) \, {\overset{?} =} \, \frac{32}{\pi^3}. 
\end{equation}
The formula $\rm H(1/2)$ in \cite[Section 4]{guilleraEMpi2} is the ``upside-down" of the following ``divergent" Ramanujan-like series for $1/\pi^3$:
\begin{equation}
\sum_{n=0}^{\infty}\frac{\left(\frac12 \right)_n^7}{(1)_n^7}\, 2^{6n} \,(84n^3 +88n^2+32n+4) \, ``{\overset{?} =}" \, \frac{-24 i}{\pi^3}.
\end{equation}
Yue Zhao (2017) has discovered the following ``divergent" Ramanujan-like series for $1/\pi^3$. 
\begin{equation}
\sum_{n=0}^{\infty}\frac{\left(\frac12 \right)_n^5\left(\frac13 \right)_n\left(\frac23 \right)_n }{(1)_n^7}\, \left( \frac{27}{4} \right)^n \,(92n^3 +84n^2+27n+3) \, ``{\overset{?} =}" \, \frac{-48 i}{\pi^3}.\end{equation}
The upside-down of it is a convergent formula for $\pi^4$ given in \cite{Zhao}. The formula (\ref{cullen}) was discovered by Jim Cullen in the year $2010$:
\begin{equation}\label{cullen}
\sum_{n=0}^{\infty}  \frac{\left(\frac{1}{2}\right)_n^7 \left(\frac{1}{4}\right)_n \left(\frac{3}{4}\right)_n}{(1)_n^9} \frac{1}{2^{12n}}  (43680n^4+20632n^3+4340n^2+466n+21) \, {\overset{?}=} \, \frac{2048}{\pi^4}.
\end{equation}
The following two formulas have been discovered by Yue Zhao \cite{Zhao}:
\begin{equation}
\sum_{n=0}^{\infty}  \frac{\left(\frac12 \right)_n^5 \left(\frac13 \right)_n \left(\frac23 \right)_n\left(\frac14 \right)_n \left(\frac34 \right)_n}{(1)_n^9} \left( \frac{-27}{256} \right)^n  (4528n^4+3180n^3+972n^2+147n+9) \, {\overset{?}=} \, \frac{768}{\pi^4}.
\end{equation}
\begin{equation}
\sum_{n=0}^{\infty}  \frac{\left(\frac12 \right)_n^5 \left(\frac15 \right)_n \left(\frac25 \right)_n\left(\frac35 \right)_n \left(\frac45 \right)_n}{(1)_n^9} \left( \frac{-5}{4} \right)^{5n}  (5532n^4+5600n^3+2275n^2+425n+30) ``\, {\overset{?}=} \," \frac{1280}{\pi^4}.
\end{equation}
The last one is ``divergent" and Zhao gives in \cite{Zhao} the udside-down of it: a remarkable formula for $\zeta(5)$. All these formulas satisfy supercongruences of Zudilin type (that is, which are analogues to those in  \cite{Zu0} and \cite{GuiZu}).

\section{New conjectures (Added in June 2018)}

We will use the notation: $s_0=1/2$, $s_3=1-s_1$, and $s_4=1-s_2$, and the following couples for $(s_1, s_2)$: $(1/2,1/2)$, $(1/2,1/3)$, $(1/2,1/4)$, $(1/2,1/6)$, $(1/3,1/3)$, $(1/3,1/4)$, $(1/3,1/6)$, $(1/4,1/4)$, $(1/4,1/6)$, $(1/6,1/6)$, $(1/5,2/5)$, $(1/8,3/8)$, $(1/10,3/10)$, $(1/12,5/12)$. 

\subsection{Case $_5F_4$}
Let
\[
{\rm B_5F_4}(s_1,s_2,z,x)=\omega(s_1,s_2,x) \sum_{n \in \mathbb{Z}} (-1)^n \left[\prod_{i=0}^{4} \frac{(s_i)_{n+x}}{(1)_{n+x}} \right] (-z)^{n+x},
\]
where 
\[
\omega(s_1,s_2,x)=\cos \pi x \, \frac{\cos^2 \pi x - \cos^2 \pi s_1}{\sin^2 \pi s_1} \, \frac{\cos^2 \pi x - \cos^2 \pi s_2}{\sin^2 \pi s_2},
\]
is the factor we need to have an analytic function. If $z>0$ we replace $(-1)^n (-z)^{n+x}$ with $e^{-i \pi x} z^{n+x}$. We know that there exist $u_0$, $u_1$, $u_2$, $v_1$, $v_2$, such that the following Fourier expansion holds:
\[
{\rm B_5F_4}(s_1,s_2,z,x) = u_0+u_1 \cos 2\pi x +u_2 \cos 4\pi x + v_1 \sin 2 \pi x + v_2 \sin 4 \pi x.
\] 
$ \rm B_5F_4(s_1,s_2,z,x)$ is calculated by the Maple procedure \texttt{bilater(s1,s2,z0,x,0)} of page $3$. Based on numerical calculations we make the following conjecture: 
\subsection*{}
For the special values of $z$ (those that appear in the Ramanujan-like series for $1/\pi^2$) with $z<0$, there is an integer relation among $u_0$, $u_1$, $u_2$,
and 2-) For the special values of $z$ with $z>0$, there is an integer relation among $Re(u_0)$, $Re(u_1)$, $Re(u_2)$. There is also another relation: $ Im(u_0+u_1+u_2)=0$, however this last relation holds for all the values of $z$ and therefore is not exclusive of the special values of $z$.
\\ \par
Examples: Executing \texttt{bilater(s1,s2,z0,x,0)} for five values of $x$, solving numerically a system of equations, and using the PSLQ algorithm we get examples like the following ones: For $B_5F_4(1/2, 1/2, -1/4, x)$ we get the relation $u_0+u_1+9u_2=0$; for $B_5F_4(1/2, 1/2, -1/2^{10}, x)$ we obtain $11 u_0+15 u_1+ 35 u_2=0$; and for $B_5F_4(1/3, 1/4, -1/48, x)$ we get the relation $17u_0+23 u_1+ 65 u_2=0$. The following two examples are for series of positive terms: For $B_5F_4(1/8, 3/8, 1/7^4, x)$, we get $Re(37u_0+43u_1+69u_2)=0$; and for $B_5F_4(1/2, 1/3, (3/\phi)^3, x)$ we get $Re(35u_0+41u_1+131u_2)=0$, where $\phi$ is the fifth power of the golden ratio.

\subsection{Case $_4F_3$}
Let
\[
{\rm B_4F_3}(s_1,s_2,z,x)=\omega(s_1,s_2,x) \sum_{n \in \mathbb{Z}} \left[\prod_{i=1}^{4} \frac{(s_i)_{n+x}}{(1)_{n+x}} \right] z^{n+x},
\]
where 
\[
\omega(s_1,s_2,x)=\frac{\cos^2 \pi x - \cos^2 \pi s_1}{\sin^2 \pi s_1} \, \frac{\cos^2 \pi x - \cos^2 \pi s_2}{\sin^2 \pi s_2},
\]
is the factor we need to have an analytic function. If $z<0$, we replace $z^{n+x}$ with $e^{-i \pi x} (-z)^{n+x}$. We know that there exist $u_0$, $u_1$, $u_2$, $v_1$, $v_2$, such that the following Fourier expansion holds:
\[
{\rm B_4F_3}(s_1,s_2,z,x) = u_0+u_1 \cos 2\pi x +u_2 \cos 4\pi x + v_1 \sin 2 \pi x + v_2 \sin 4 \pi x.
\] 
For the two special values of $z$ given in \cite[Table 14]{watkins}, we observe numerically that for $B_4F_3(1/3, 1/3, -1/8, x)$ we have the relations 
$Re(2u_0-5u_1)=0$, $Im(u_0+u_1)=0$ and $Re(3u_0+5u_2)=0$, $Im(u_2)=0$; and for $B_4F_3(1/3, 1/4, -4, x)$ we see that $Re(6u_0-7u_1)=0$, $Im(2u_0+3u_1)=0$, and $Re(u_0+7u_2)=0$, $Im(u_0+3u_2)=0$.

\begin{remark}
We can state other conjectures of the same nature for hypergeometric series of types $_6F_5$, $_7F_6$, etc.
\end{remark}

\end{document}